 \documentclass[final,3p,times]{elsarticle}

\usepackage{amssymb}
\usepackage{amsmath}
\usepackage{amsthm}

\newcommand{\R}{\mathcal{R}}
\newcommand{\Z}{\mathcal{Z}}
\newcommand{\F}{\mathcal{F}}
\newcommand{\PA}{\mathcal{P}}
\newcommand{\M}{\mathcal{M}}
\newcommand{\I}{\mathcal{I}}
\newtheorem{theorem}{Theorem}[section]
\newtheorem{lemma}[theorem]{Lemma}
\newtheorem{remark}{Remark}

\journal{Operations Research Letters}

\begin{document}

\begin{frontmatter}



\title{Stochastic Approximation with Two Time Scales: The General Case}

\author[VB]{Vivek S. Borkar}
\ead{borkar.vs@gmail.com}
\address[VB]{Department of Electrical Engineering,
Indian Institute of Technology Bombay, Mumbai 400076, India}


\begin{abstract}
 Two time scale stochastic approximation is analyzed when the iterates on either or both time scales do not necessarily converge.

\end{abstract}





\end{frontmatter}



\section{Introduction}

Recall the two time scale stochastic approximation in $\R^d\times\R^s$ given by (see  \cite{Borkar2T}, also section 8.1 of \cite{BorkarBook}) 
\begin{equation*}
\begin{aligned}
x(n+1) &= x(n) + a(n)[h(x(n),y(n)) + M(n+1)], \\
y(n+1) &= y(n) + b(n)[g(x(n),y(n)) + M'(n+1)],
\end{aligned}
\end{equation*} 
where:
\begin{enumerate}
\item for $d,s\geq 1$, $h: \R^d\times\R^s \mapsto \R^d$, $g: \R^d\times\R^s \mapsto \R^s$ are Lipschitz,

\item $\{M(n)\}, \{M'(n)\}$ are martingale difference sequences with respect to the increasing $\sigma$-fields 
$$\F_n := \sigma(x(0),y(0),M(m),M'(m), m\leq n),$$
satisfying
\begin{equation}
E\left[\|M(n+1)\|^2 + \|M'(n+1)\|^2 | \F_n\right]  \leq \mathcal{K}\left(1 + \|x(n)\|^2 + \|y(n)\|^2\right) \label{mgbd}
\end{equation}
for $n\geq 0$, and,
\item  $a(n), b(n) \in (0,\infty)$ are step size sequences satisfying the Robbins-Monro conditions 
$$\sum_na(n) = \sum_nb(n) = \infty; \sum_n(a(n)^2+b(n)^2)<\infty,$$ 
and the additional requirement $b(n) = o(a(n))$.
\end{enumerate}

We shall assume that these iterates are stable, i.e.,
\begin{equation}
\sup_n(\|x(n)\|+\|y(n)\|)<\infty, \ \mbox{a.s.} \label{stable}
\end{equation}
This usually needs to be established separately, see, e.g., \cite{Laxmi}.\\

We shall take the `ODE approach' (for `Ordinary Differential Equation') to stochastic approximation. See \cite{Fradkov}, \cite{Ljung}, \cite{Meerkov}, for some early work, \cite{Benaim1}, \cite{Benaim2} for the state of the art, and \cite{Ben}, \cite{BorkarBook} for a textbook treatment. In this approach,  one views the above iterations as noisy discretizations of differential equations 
$$\dot{x}(t) = h(x(t),y(t)), \ \dot{y}(t) = g(x(t),y(t)),$$ 
resp., except that the condition $b(n) = o(a(n))$ implies that the latter ODE moves on a slower time scale. This makes the situation akin to the `singularly perturbed ODEs' 
$$\dot{x}(t) = h(x(t),y(t)), \ \dot{y}(t) = \epsilon g(x(t),y(t))$$ 
in the $0 < \epsilon\downarrow 0$ limit. Following the standard philosophy for analyzing such ODEs, $x(\cdot)$ sees $y(\cdot)$ as \textit{quasi-static}, i.e., with $y(t) \approx$ a constant $y$. Thus $\dot{x}(t) \approx h(x(t),y)$. Suppose the ODE $\dot{x}(t) = h(x(t),y)$
has a unique asymptotically stable equilibrium $\lambda(y)$ for a Lipschitz $\lambda(\cdot)$. Being on a slower time scale, $y(\cdot)$ in turn sees $x(\cdot)$ as \textit{quasi-equilibrated}, i.e., $x(t) \approx \lambda(y(t))$. Hence $\dot{y}(t) \approx g(\lambda(y(t)), y(t))$. Suppose the ODE $\dot{y}(t) = g(\lambda(y(t)), y(t))$ has a unique globally asymptotically stable equilibrium $y^*$. Then one expects $(x(n), y(n)) \to (\lambda(y^*),y^*)$ a.s., which is indeed the case (\cite{Borkar2T}, section 8.1 of \cite{BorkarBook}). In applications, the faster iterates often emulate a subroutine of the algorithm on a slower time scale, albeit concurrently updated. Because of this, the two time scale stochastic approximation has had many applications, see, e.g., \cite{Karmakar}, \cite{Konda}.

There are, however, situations when we do have a two time scale stochastic approximation scheme, but \textit{sans} unique asymptotically stable equilibria as above. Perhaps the simplest such instance is that of multiple equilibria in case one of the iterations is a stochastic gradient descent for a non-convex function. Other important examples can be found in reinforcement learning, e.g., \cite{Yaji}. This motivates the present study, which aims to derive a broad characterization of the asymptotic behavior of two time scale algorithms in the spirit of \cite{Benaim1}, \cite{Benaim2}.

\section{Preliminaries}

We shall consider more general iterates
\begin{equation*}
\begin{aligned}
x(n+1) &= x(n) + a(n)[h(x(n),y(n), Z(n))  + \ M(n+1)], \\
y(n+1) &=  y(n) + b(n)[g(x(n),y(n), Z(n))  + \ M'(n+1)],
\end{aligned}
\end{equation*}
where $\{Z(n)\}$ is the so called Markov noise taking values in a finite set $\Z$  and satisfying the condition: for 
$$\F_n := \sigma(x(0), y(0), M(m), M'(m), Z(m), m \leq n), n \geq 0,$$
one has
$$P(Z(n+1) \in A| \F_n) = p_{x(n),y(n)}(A|Z(n)), n \geq 0,$$
for any Borel set $A \in \Z$ and a parametrized transition kernel $p_{x,y}(\cdot|z)$ such that the map $(x,y) \in \R^d\times\R^s \mapsto p_{x,y}(\cdot|z) \in \PA(\Z)$\footnote{$\PA(\Z) :=$ the $|\Z|$-dimensional probability simplex. Here and elsewhere, we shall denote by $\PA(\cdots)$ the space of probability measures on a Polish space `$\cdots$' with the Prohorov topology.} is Lipschitz.  We assume that the transition probability function $p_{x,y}(\cdot | \cdot)$ is irreducible $\forall \ x,y,$ and denote by $\pi_{x,y} \in \PA(\Z)$ its unique stationary distribution. Since $\pi_{x,y}$ is a vector of rational functions of the transition probabilities with non-vanishing denominators by Cramer's thorem, it will also be Lipschitz in $x,y$.

\begin{remark} We take $\Z$ to be finite for notational simplicity. More general state spaces or the Markov noise are possible, see the comments in the concluding section. \end{remark}

Define $\tau(0) = t(0) = 0, \tau(n+1) = \tau(n) + a(n), t(n+1) = t(n) + b(n),$ for $n \geq 0$. Define continuous and piecewise linear interpolations $\bar{x}(t), \bar{y}(t), t \geq 0,$ of $\{x(n)\}, \{y(n)\}$ resp.\ by:
\begin{equation*}
\begin{aligned}
\bar{x}(t) &= \Big( \frac{t-\tau(n)}{\tau(n+1)-\tau(n)} \Big) x(n)   + \Big( \frac{\tau(n+1)-t}{\tau(n+1)-\tau(n)} \Big) x(n+1)  \mbox{ for} \ \tau(n) \leq t \leq \tau(n+1), \\
\bar{y}(t) &= \Big( \frac{t-t(n)}{t(n+1)-t(n)}\Big) y(n)   + \Big( \frac{t(n+1)-t}{t(n+1)-t(n)} \Big) y(n+1), 
\mbox{ for} \ t(n) \leq t \leq t(n+1),
\end{aligned}
\end{equation*}
for $n \geq 0$. Fix $T > 0$. Define also solutions $x^n(t), \tau(n) \leq t \leq \tau(n) + T$,  of the ODEs 
\begin{equation}
\dot{x}_y^n(t) = \sum_z h(x^n_y(t), y, z)\pi_{x^n_y(t),y}(z), \ x^n_y(\tau(n)) = x(n), 
 t \in [\tau(n), \tau(n)+T], \label{ODE1}
\end{equation}
for $y= y(n) \in \R^s$ treated as a constant parameter. 

Then we have the following result. (See, e.g., \cite{Ben} or  sections 8.2-8.3 of \cite{BorkarBook}.)

\begin{lemma} Almost surely,
\begin{equation}
\lim_{n\to\infty}\sup_{t\in[t(n),t(n)+T]}\|\bar{x}(t) - x^n(t)\| = 0. \label{xlim}
\end{equation}
\end{lemma}

\begin{proof}{({\textit{Sketch}})} 
This follows by standard arguments for stochastic approximation with Markov noise, see, e.g., \cite{Ben} and  sections 8.2-8.3 of \cite{BorkarBook} for two alternative approaches. The only additional feature is the presence of $y(n)$ in the dynamics of $x^n(\cdot)$, which is the iterates $y(\cdot)$ on the slower time scale kept frozen at $y(n)$ according to the usual two time scale logic. \end{proof}

This will play a key role in the proof of our main result, Theorem \ref{fast} below.

\section{Asymptotics for the fast time scale}

Our arguments will be pathwise, a.s. Consider a sample path in the probability $1$ set $\Omega_0$ where \eqref{stable} holds. Thus, in particular, $C_X := \sup_n\|x(n)\| < \infty$ for this sample path. (Note that $C_X$ is random.) Let $B_X := \{x \in \R^d : \|x\| \leq C_X\}$. Let $m(n) := \min\{k \geq n : t(k) \geq t(n)+T\}$. Define a $\PA(\R^d\times\Z)$-valued process $\mu_t(\cdot), t \in [t(n),t(n)+T]$ by:
\begin{equation*}
\mu_t(A\times D) := \delta_{x(n+k)}\delta_{Z(n+k)} 
\mbox{ for} \ t(n+k) \leq t < t(n+k+1)\wedge T, 0 \leq k <  m(n),
\end{equation*}
where $\delta_x$ stands for the Dirac measure at $x$. Consider $\mu_{t(n) + \cdot}$ as a random variable taking values in the space $\M :=$ the space of measurable maps $[0,\infty) \mapsto \PA(\R^d\times\Z)$ with the coarsest topology that renders continuous the maps
$\mu \in \M \mapsto \int_0^T g(t)\int_{B_X\times\Z} fd\mu_tdt$ for any $f \in C(B_X\times\Z)$, $g \in L_2[0,T]$, and $T > 0.$ 
Then $\M$ is a compact Polish space, see, e.g., Lemma 5.3, p.\ 71, of \cite{BorkarBook} (see also \cite{Yuksel}).
Let $\mu_\cdot^*$ be any limit point in $\M$ of  $\mu_{t(n) +  \cdot}$ as $n\to\infty$. By dropping to a further subsequence if necessary, let $x(n) \to x^*, y(n) \to y^*$ along this subsequence. Our first key result is:

\begin{theorem}\label{fast} Almost surely, $\mu_t^*$ is of the form
$\mu^*_t(dx,z) = \eta^*(dx)\pi_{x,y}(z)$ where $\eta(\cdot)$ belongs to the compact convex set $\I_*$ of invariant distributions of the ODE
\begin{equation}
\dot{x}^*(t)  = h^*(x^*(t), y^*)  
 := \sum_z h(x^*(t), y^*, z)\pi_{x^*(t),y^*}(z), \ t \geq 0. \label{limODE}
\end{equation}
\end{theorem}

\begin{remark}
Note that the ODE \eqref{limODE} is well posed because the map $x \mapsto  \sum_z h(x, y^*, z)\pi_{x,y^*}(z)$ is Lipschitz. \end{remark}

\begin{proof} Consider a fixed sample path as above. Fix $T>0$. Let $f$ be a smooth compactly supported function on $B_X$. Let 
\[
m(n) := \min\{k \geq n : \sum_{\ell = n}^k t(\ell) \geq  t(n) + T\}.
\]
Then, since $\sum_{k=n}^{m(n)}a(k) \approx T$, we have
\begin{align}
| f(x(t(m(n)))    - f(t(n))|   &  \leq \max_{w\in B_X}\|\nabla f(w)\|\left(\left\|\sum_{k=n}^{m(n)}b(k)x(k)\right\| \right)  \nonumber \\
& \leq    \max_{w\in B_X}\|\nabla f(w)\|\Bigg(\max_{\{n\leq k \leq m(n)\}} \left(\frac{b(k)}{a(k)}\right)\times 
 \max_k\|x(k)\|\left\|\sum_{k=n}^{m(n)}a(k)\right\|\Bigg)  \nonumber \\
&\leq  K\max_{\{n\leq k \leq m(n)\}}\left(\frac{b(k)}{a(k)}\right) \to 0 \label{gotozero}
\end{align}
almost surely, where $K > 0$ is a possibly random finite constant. On the other hand, using the first order Taylor expansion, we have
\begin{align}
f(x(t(m(n)))   - f(t(n))   &=  \sum_{k=n}^{m(n)}b(k)\langle\nabla f(x(k)), x(k+1) - x(k) \rangle + o(n)  \nonumber \\
&= \sum_{k=n}^{m(n)}b(k)\langle\nabla f(x(k)), h(x(k),y(k),Z(k)) \rangle  + \ \sum_{k=n}^{m(n)}b(k)\langle\nabla f(x(k)), M(k+1) \rangle + o(n). \label{bound1}
\end{align}
Defining
$$W(n) := \sum_{m=0}^{n-1}b(k)\langle\nabla f(x(k)), M(k+1) \rangle, \ n \geq 0,$$
$(W(n), \F_n)$ is seen to be a square integrable martingale with  quadratic variation
\begin{equation*}
\langle W \rangle(n) \leq C\left(1+\sup_\ell\|x(\ell)\|^2 + \|y(\ell)\|^2\right)  \left(\sum_{k=0}^\infty b(k)^2\right) 
 < \infty \ \ \mbox{a.s.}
\end{equation*}
for a suitable constant $C > 0$, in view of \eqref{mgbd}. Hence by \eqref{stable} and Proposition VII.3.2(a), p.\ 149, of \cite{Neveu}, $W(n)$ converges a.s.\ as $n\to\infty$. 
Therefore the second term on the right in \eqref{bound1} converges to zero, a.s. By enlarging the zero probability set $\Omega^c_0$ if necessary, we assume that this holds true for the chosen sample path. On the other hand, the left hand side of \eqref{gotozero} tends to $0$ a.s.\ as $n\to\infty$ by \eqref{gotozero}. Thus the first term on the right of \eqref{bound1} also tends to zero a.s.\ as $n\to\infty$. In view of our definition of $\{\mu_t\}$, it then follows that
$$\int_{t(n)}^{t(n) + T}\int_{\R^{d}}\sum_{z\in\Z}\langle \nabla f(x), h(x,y(t),z)\rangle\mu_t(dx,z)dt \to 0.$$
It follows that outside a set of zero probability, every limit point $(\mu^*_{\cdot}, y^*(\cdot))$ of $(\mu_{t+\cdot}, y(t+\cdot))$ in $\M\times C([0,\infty);\R^s)$ satisfies:
$$\int_0^T\int_{\R^{d}}\sum_{z\in\Z}\langle \nabla f(x), h(x,y^*(t),z)\rangle\mu^*_t(dx,z)dt = 0.$$
Since $T > 0$ was arbitrary, we can conclude that
$$\int_{\R^{d}}\sum_{z\in\Z}\langle \nabla f(x), h(x,y^*(t),z)\rangle\mu^*_t(dx,z) = 0 \ \forall \ t.$$
Setting $y(t) = y$, disintegrate $\mu^*_t$ as 
$$\mu^*_t(dx,z) = \eta_t(dx)\pi_{x,y}(z).$$
Here the fact that the regular conditional law is precisely $\pi_{x,y}(\cdot)$ follows by direct verification from the fact that $\pi_{x(n),y(n)}(\cdot)$ is the regular conditional law of $Z(n)$ given $x(n), y(n)$ for all $n$, and the map $(x,y) \mapsto \pi_{x,y}(\cdot)$ is continuous. 
From Theorem 4.1 of \cite{Stockbridge}, it follows that $\eta_t$ is a stationary distribution for the ODE
$$\dot{x}_y(t') = \sum_{z\in\Z}h(x_y(t'), y,z)\pi_{x(t'),y}(z), \ t' \geq 0,$$
where we have kept $t \geq 0$ and $y = y(t)$ fixed. This proves the claim. \end{proof}

\begin{remark}\label{Stock}  Controlled ODEs are a rather special and degenerate case of the much more general formalism of \cite{Stockbridge}. \end{remark}
\section{Asymptotics for the slow time scale}

Define the solutions $y^t(s), t \leq s \leq t + T$,  of the ODEs 
\begin{equation}
\dot{y}^t(s) = \int_{\R^d}\sum_{z \in \Z} g(x, y^t(s), z)\mu_s(dx,z), \ s\in [t, t+T]. \label{ODE2}
\end{equation}
  In view of the foregoing, the following is immediate.
  
 \begin{theorem} Almost surely, 
 $$\lim_{t\uparrow\infty}\sup_{s\in [t, t + T]}\|\bar{y}(s) - y^t(s)\| = 0 \ \mbox{a.s.,}$$
 where $y^t(\cdot)$ is a solution to the ODE
\begin{equation}
\dot{y}^t(s) = \int_{\R^{d_1}}\sum_{z\in\Z}g(x,y^t(s),z)\eta_{y^t(s)}(dx), \ t \in[t, t+T]. \label{yODE}
\end{equation}
\end{theorem}

\begin{proof} Note that $\bar{y}(\cdot)$ satisfies
\begin{equation*}
\begin{aligned}
\bar{y}(n  + k)   &  = \bar{y}(t(n)) + \sum_{m=0}^{k-1}b(n+m )  \Big(g(\bar{x}(t(n) + m), \bar{y}(t(n) + m), Z(n+m))  
+ \ M'(m+1)\Big) \\
& = \bar{y}(t(n))   + \int_{t(n)}^{t(n)+k}\int_{R^{d}}\sum_{z\in\Z}g(x,\bar{y}(n),z) \mu_t(dx,z)dt  + o(n)
\end{aligned}
\end{equation*}
for $0 \leq k <t(m(n))-t(n)$. Passing to the limit as $n\to\infty$ along a suitable subsequence, the characterization of the limiting ODE follows in view of Theorem \ref{fast} above. 
\end{proof}
 
 The catch here is that the probability measure $\eta$ in \eqref{limODE} may not be unique. We do, however, have the following.
 
 \begin{lemma} The set of invariant probability measures $J_y$ for \eqref{limODE} is nonempty compact and convex for each fixed $y \in \R^s$. In addition, the set valued map $y \mapsto J_y \subset \PA(\R^d)$ is upper semicontinuous in the sense that its graph $\{(y, \eta_y) : y \in \R^s\}$ is closed in $\R^s\times\PA(R^d)$. \end{lemma}
 
 \begin{proof} The  fact that there exists at least one invariant probability measure is already contained in the proof of Theorem \ref{fast}. On the other hand, the set of $\eta \in \PA(\R^d)$ satisfying
 the equation
 \begin{equation}
 \int_{\R^{d}}\sum_{z\in\Z}\langle \nabla f(x), h(x,y,z)\rangle\pi_{x,y}(z)\eta(dx) = 0, \label{Ech}
 \end{equation}
is closed and convex because the equation is preserved under convex combinations and convergence in $\PA(\R^d)$. The first claim now follows from the main theorem of \cite{Echeverria}. Let $y_n\to y_\infty$ in $\R^s$ and let $\eta_{y_n} \in J_{y_n}, n \ge 1$. Then setting $y = y_n$  in \eqref{Ech} and letting $n\uparrow\infty$, any limit point $\eta$ of $\eta_{y_n}$ as $n\uparrow\infty$ is seen to satisfy \eqref{Ech} with $y = y_\infty$ and therefore is an invariant probability measure for \eqref{limODE} with $y = y_\infty$. This completes the proof.  \end{proof}

\begin{remark} A remark similar to Remark \ref{Stock}  applies here  vis-a-vis ODEs and the results of \cite{Echeverria}. \end{remark}

In view of the potential non-uniqueness of $\eta_y$, we replace \eqref{yODE} by the differential inclusion
\begin{equation}
\dot{y}(t) = H(y(t)) \label{DiffIncl}
\end{equation}
 where 
 $$H(y) := \left\{\int_{\R^{d_1}}\sum_{z\in\Z}\pi_{x,y}(z)g(x,y,z)\eta(dx) : \eta \in J_y\right\}.$$
It is easy to see that the set-valued map $y \in \R^s \mapsto H(y) \subset \R^s$ is nonempty closed and convex valued because the set-valued map $y \in \R^s \mapsto J_y \subset \PA(\R^d)$ is.

We recall now some concepts related to differential inclusions from \cite{Benaim3}. A trajectory $y(\cdot)$ of \eqref{DiffIncl} is an absolutely continuous function $\R\mapsto \R^s$ such that $\dot{y}(t) = \kappa(t), t \in \R$ for some measurable $\kappa:\R\mapsto\R^s$ satisfying $\kappa(t)\in H(y(t))$ a.e. A set $A\subset \R^s$ is said to be invariant for \eqref{DiffIncl} if for any $x\in A$, there exists a trajectory $y(\cdot)$ of \eqref{DiffIncl} passing through $x$ such that $y(t) \in A \ \forall t\in\R$. A compact invariant set $A\subset\R^s$ is said to be internally chain transitive if given any $x,y\in A$ and any $\epsilon, T > 0$, one can find $x_i \in A, t_i > T, 1 \leq i \leq n,$ such that there exist trajectories $y^i(t), t \in [0,t_i],$ in $A$ for $1\leq i < n$ such that $\|y^i(0)-x_i\|<\epsilon$ and $\|y^i(t_i)-x_{i+1}\| < \epsilon$ for $1\leq i < n$. In this framework, our main result is the following.

\begin{theorem} Almost surely, as $t\to\infty$, every limit point of $\bar{y}(t + \cdot )$ in $C((-\infty,\infty);\R^s)$ is a trajectory in an internally chain transitive invariant set of \eqref{DiffIncl}. \end{theorem}

\begin{proof} This follows from Theorem 4.3 of \cite{Benaim3}. 
\end{proof}

\section{Extensions and future directions}

An obvious extension that is desirable for applications is to have a more general state space for the `Markov noise' process $\{Y_n\}$. Comparing with sections 8.2, 8.3 of \cite{BorkarBook}, it is clear that what this would require at the least is tightness of the laws of $\{Y_n\}$ and some regularity of the dependence of the transition kernel $p_{x,y}(\cdot|\cdot)$ on $x,y$ that reflects into a similar regularity of the corresponding stationary distribution, assuming it is unique. Non-uniqueness of the latter adds further complications.

As for future directions, here are some speculations. Given the existing work on small noise limits, it stands to reason that one would expect that one can restrict to a proper subset of $J_y$ and therefore of $H(y)$ in the foregoing. For example, the Freidlin-Wentzell theory  \cite{Freidlin} suggests that small noise limits of invariant probability distributions should concentrate on stable attractors of \eqref{limODE} that minimize the Freidlin-Wentzell quasi-potential. Unfortunately that seems too ambitious here. The reason is that the noise is added on the same time scale as the drift, i.e., on the time scale defined by the stepsizes $\{a(n)\}$. To get a stochastic differential equation limit, it would have to be weighted by $\sqrt{a(n)}$, and to get concentration on minimizers of the quasi-potential, it should be on an even slower time scale correwsponding to much more slowly decreasing weights, as suggested by the analysis in the special case of gradient systems in \cite{Gelfand}. Thus pinning down a selection principle for invariant probability measures remains a challenge here, though some very partial results are available (see, e.g., Chapter 3 of \cite{BorkarBook}). Such `Kolmogorov measures' are also of interest in dynamical systems theory \cite{Ruelle} and these connections need to be further explored, both in the present context and in other related contexts such as \cite{Bianchi}, \cite{BorkarShah}.





\begin{thebibliography}{00}

\bibitem{Benaim1} Benaim, M., 1996. ``A dynamical system approach to stochastic approximations", \textit{SIAM Journal on Control and Optimization}, 34(2), 437-472.


\bibitem{Benaim2} Benaïm, M., 2006. ``Dynamics of stochastic approximation algorithms", In \textit{Seminaire de Probabilites XXXIII}, Springer, 1-68.

\bibitem{Benaim3} Benaïm, M., Hofbauer, J.\ and Sorin, S., 2005. ``Stochastic approximations and differential inclusions", \textit{SIAM Journal on Control and Optimization}, 44(1), 328-348.

\bibitem{Ben} Benveniste, A., M\'{e}tivier, M.\ and Priouret, P., 1990. \textit{Adaptive {A}lgorithms and {S}tochastic {A}pproximations,} Springer Verlag.

%
\bibitem{Bianchi} Bianchi, P.\ and  Rios-Zertuche, R., 2024. ``A closed-measure approach to stochastic approximation", Stochastics, 1-23.

\bibitem{Borkar2T} Borkar, V.\ S., 1997. ``Stochastic approximation with two time scales", \textit{Systems and Control Letters} 29(5), 291-294.

\bibitem{BorkarCM} Borkar, V.\ S., 2006. ``Stochastic approximation with `controlled Markov' noise", \textit{Systems and Control Letters} 55(2), 139-145.

\bibitem{BorkarBook} Borkar, V.\ S., 2022/24. \textit{Stochastic {A}pproximation: {A} {D}ynamical {S}ystems {V}iewpoint} (2nd ed.), Hindustan Publishing Company and Springer Nature.

\bibitem{BorkarShah} Borkar, V.\ S.\ and Shah, D.\ A., 2024. ``Remarks on differential inclusion limits of stochastic approximation", \textit{Pure and Applied Functional Analysis} 3, to appear (also, arXiv preprint arXiv:2303.04558).

\bibitem{Fradkov} Derevitskii, D.\ P.\ and Fradkov, A.\ L.\ V., 1974. ``Two models analyzing the dynamics of adaptation algorithms", \textit{Avtomatika i Telemekhanika} 1, 67-75.

\bibitem{Ruelle} Eckmann, J.-P.\ and Ruelle, D., 1985. ``Ergodic theory of chaos and strange attractors", \textit{Reviews of Modern Physics} 57(3) Part 1, 617-656.

\bibitem{Freidlin} Freidlin, M.\ I.\ and Wentzell, A.\ D., 20102. \textit{Random Perturbations of Dynamical Systems (3rd edition)}, Springer Verlag.

\bibitem{Echeverria} Echeverria, P., 1982. ``A criterion for invariant measures of Markov processes", \textit{Zeitschrift f\"{u}r Wahrscheinlichkeitstheorie verw Gebiete} 61, 1–16. 

\bibitem{Gelfand} Gelfand, S.\ B.\ and Mitter, S.\ K., 1991. ``Recursive stochastic algorithms for global optimization in $R^d$", \textit{SIAM Journal on Control and Optimization} 29(5), 999-1018.

\bibitem{Karmakar} Karmakar, P.\ and Bhatnagar, S., 2018. ``Two time-scale stochastic approximation with controlled Markov noise and off-policy temporal-difference learning", \textit{Mathematics of Operations Research}, 43(1), 130-151.

\bibitem{Konda} Konda, V.\ R.\ and Borkar, V.\ S., 1999. ``Actor-critic--type learning algorithms for Markov decision processes", \textit{SIAM Journal on control and Optimization}, 38(1), 94-123.

\bibitem{Laxmi} Laxminarayanan, C.\ and Bhatnagar, S., 2017. ``A stability criterion for two timescale stochastic approximation schemes", \textit{Automatica} 79, 108-114.

\bibitem{Ljung} Ljung, L., 1977. ``Analysis of recursive stochastic algorithms", \textit{IEEE Transactions on Automatic Control}, 22(4), 551-575.

\bibitem{Meerkov} Meerkov, S.\ M., 1972. ``Simplified description of slow random walks II", \textit{Automation and Remote Control} 33(2), 403-414.


\bibitem{Neveu} Neveu, J., 1975. \textit{Discrete-Parameter Martingales}, North Holland / American Elsevier.

\bibitem{Stockbridge} Stockbridge, R.\ H., 1990. ``Time-average control of martingale problems: existence of a stationary solution", \textit{The Annals of Probability} 18(1), 190-205.

\bibitem{Yaji} Yaji, V.\ G.\ and Bhatnagar, S., 2020. ``Stochastic recursive inclusions in two timescales with nonadditive iterate-dependent Markov noise", \textit{Mathematics of Operations Research} 45(4), 1405-1444.

\bibitem{Yuksel} Yüksel, S., 2024. ``On Borkar and Young relaxed control topologies and continuous dependence of invariant measures on control policy", \textit{SIAM Journal on Control and Optimization}, 62(4), 2367-2386.




%
%
%
\end{thebibliography}


\end{document}